\documentclass[11pt,reqno]{amsart}
\usepackage{amsmath,amsfonts,amscd,amsxtra,amssymb,latexsym,calc,xfrac,geometry}

\usepackage[a4paper,bookmarks,
bookmarksnumbered,%
 colorlinks=true,%
 linkcolor=blue,%
 citecolor=blue,%
 filecolor=blue,%
 urlcolor=blue,%
]
{hyperref}

 \usepackage[T1]{fontenc}
 \usepackage[latin1]{inputenc}
 \usepackage{times}
\usepackage{euscript} 
 
\def\cal{\mathcal} 
\def\wt{\widetilde}
\def\ol{\overline}
\def\wh{\widehat}

\def\CC{\mathbb{C}}
\def\PP{\mathbb{P}}
\def\QQ{\mathbb{Q}}
\def\ZZ{\mathbb{Z}}
  
\def\RR{\mathbb{R}}

\def\OO{\cal O}

\def\s-{\setminus}

\def\bcp{\mathbb C\mathbb P}
\def\cpb{\overline{\mathbb C \mathbb P ^2}}

\newtheorem{main}{Theorem}
\newtheorem{propo}[main]{Proposition}

\newtheorem{thm}{Theorem}[section]
\newtheorem{prop}[thm]{Proposition}

\newtheorem{defn}[thm]{Definition}

\newtheorem{cor}[thm]{Corollary}

\newtheorem{rmk}[thm]{Remark}

\numberwithin{equation}{section}

\begin{document}

\title{The Yamabe invariant of a class of symplectic 4-manifolds}

\author{ I. {\c S}uvaina}

\address{
        Department of Mathematics,1326 Stevenson Center, Vanderbilt University, Nashville, TN, 37240}

\email{ioana.suvaina@vanderbilt.edu}



\date{\today}

\begin{abstract}
We compute the Yamabe invariant for a class of symplectic $4-$manifolds of general type obtained 
by taking the rational blowdown of K\"ahler surfaces.
In particular,  for any point on the half-Noether line we exhibit a simply connected minimal symplectic 
manifold for which we compute the Yamabe invariant.
\end{abstract}

\maketitle


\section{Introduction}
\label{intro}

The four dimensional case plays a special role in the smooth topology of manifolds. 
Unlike in other dimensions, a given topological $4-$manifold may admit infinitely many
smooth structures, with different geometric properties. 
The smooth structures can be distinguish via invariants, such as the Yamabe invariant.
This is, roughly speaking, the supremum of the scalar curvature of constant scalar curvature metrics with unit volume. 
For $4-$manifolds, the sign and the value of the Yamabe invariant can change when changing the smooth structure.

When the manifold admits a K\"ahler or symplectic structure of general type, the Yamabe invariant is negative. 
Moreover, the invariant also depends on the decomposability of the smooth structure, under the operation of
connected sum with complex projective spaces with the reversed orientation. 
More precisely, given a K\"ahler surface $M$ of non-negative Kodaira dimension
LeBrun shows \cite{no-ein}  that the invariant is equal to the Yamabe invariant of the minimal model $X$ of $M$ and equals:
$$Y(M)=Y(X)=-4\pi\sqrt{2c_1^2(X)}.$$
The strategy of his proof is to find first an upper bound and then show that there exists a sequence of metrics that optimizes it.
In the case when the manifold is a K\"ahler surface or a symplectic $4-$manifold of Kodaira dimension $0,1,$ or $2,$ 
LeBrun shows \cite{no-ein} that the Yamabe invariant has an upper bound in terms of the topological invariants of the minimal 
model $X$ of $M$:
\begin{equation}
\label{Y-bound}
Y(M)\leq -4\pi\sqrt{2c_1^2(X)}\leq0
\end{equation}

This gives a strong evidence that the Yamabe invariant is closely related to a topological invariant of a minimal model, 
but there is no general result in this direction. 
One has to be especially careful in the case of positive Yamabe invariant, where such a minimal model is not unique. 
For example $\bcp^2\#2\cpb$ is diffeomorphic $ (\bcp^1\times\bcp^1)\#\cpb.$ 
For these spaces, the $Y(\bcp^2)=12\sqrt 2\pi,$ \cite{Yam-plane}, but the value of the Yamabe invariant of $\bcp^1\times\bcp^1$ is still unknown.

An explicit computation of the invariant for generic spaces continues to elude us.
LeBrun raised the question when the Yamabe bound (\ref{Y-bound}) is optimized.  
Based on the fact that equality on one of the key estimates is obtained only for constant scalar curvature K\"ahler metrics, he conjectured that the Yamabe bound is optimized only in the K\"ahler case.
We show that there is a large class of symplectic manifolds of general type for which the Yamabe invariant can be computed, and it is the expected one.

\begin{main}
\label{main}
Let $M$ be a complex surface  containing $s$ linear chains $C_i$ of rational curves,
associated to the minimal resolution of a singularity of type
 $\frac1{n_i^2}(1, n_im_i-1),$ with $m_i<n_i,$ and $n_i, m_i$ relatively prime integers. 
We denote by $\wh M$  the complex orbifold surface obtained by collapsing the chains $C_i.$ 
Furthermore, let $\wh M'$ be the orbifold surface obtained from $\wh M$ by collapsing  all the rational curves 
of self-intersection $(-2)$ which do not pass through the singular points.
Assume that the canonical sheaf $K_{\wh M'}$ is ample.
Then, the symplectic manifold $N$ obtained by taking the generalized rational blowdown of the chains $C_i$ in $M$ is a minimal, symplectic $4-$manifold of general type, and its Yamabe invariant $Y(N)$ is 
\begin{equation}
Y(N)=-4\pi\sqrt{2c_1^2(N)}=Y^{orb}(\wh M)=Y^{orb}(\wh M').
\end{equation}

Moreover, if $N'=N\#l\cpb$ then $Y(N')=Y(N).$
\end{main}

\begin{rmk}
If $M$ as above is a complex surface of general type then $Y(M)<Y(N)<0.$ 
More precisely, $Y(M)^2=Y(N)^2+32\pi^2(\sum_{i=1}^s |C_i|),$ where $|C_i|$ denotes the length of the chain $C_i.$ 
Nevertheless,   $M$ can be a rational surface (see examples in \cite{lp,rs1}), 
or an elliptic surface, as the family of examples in Sect. \ref{non-K-exp}.
\end{rmk}

The geometry of symplectic $4-$manifolds is much richer than that of K\"ahler surfaces. 
One can obtain new examples by considering different types of surgeries, along symplectic
or Langrangian submanifolds. 
These surgeries can be performed within the symplectic category due to Darboux type theorems,
which allow us to establish  canonical symplectic structures in the neighborhood of these submanifolds.
Unfortunately, these techniques are not compatible with the gluing of the Riemannian structures.
In this paper, we only consider one such surgery, the generalized rational blowdown.
In Sect. \ref{constr} we present  an approach to the generalized rational blowdown surgery, 
which can be used in conjunction with a construction of preferred metrics.
This allows us  to compute the Yamabe invariant from Theorem \ref{main} in Sect. \ref{bound}.

In general, it is difficult to decide if the new symplectic manifold constructed in
Theorem \ref{main} admits an integrable complex structure. 
However, the non-existence of such an integrable complex structure follows if the numerical invariants do not satisfy either the Noether, 
or the Miyaoka-Bogomolov-Yau inequality. 
For minimal complex manifolds with even $c_1^2(M)$, the Noether inequality is $5c_1^2(M)-c_2(M)+36\geq0,$  
or equivalently $c_1^2(M)\geq2\chi_h(M)-6,$ where $\chi_h(M)$ is the Todd genus, 
and can be computed in terms of the other topological invariants as
$\chi_h(M)=\frac 1{12}(c_2(M)+c_1^2(M)).$  The minimality of the manifold implies that $c_1^2(M)>0.$
The half-Noether line is defined to be $y=x-3, y>0$  and any minimal manifold $M$ that satisfies 
$c_1^2(M)=2\chi(M)+3\tau(M)=\chi_h(M)-3$  is necessarily non-complex. 
We compute the Yamabe invariant for symplectic $4-$manifolds which lie at any point on
the half-Noether line.

\begin{main}
\label{main2}
For every point on the half-Noether line there is a simply connected, minimal, symplectic $4-$manifold of general type $M,$ for which the Yamabe invariant is
\begin{equation}
Y(M)=-4\pi\sqrt{2c_1^2(M)}.
\end{equation}
Moreover, if $M'=M\#l\cpb$ then $Y(M')=Y(M).$
\end{main}

\medskip

The above results rely on a special feature of the linear chains of curves considered.
The boundary of a neighborhood of the linear chain $C_i$ is the lens space 
$L(n_i^2, n_im_i-1).$
Essential in the construction is the existence of a rational homology ball 
$B_{n_i,m_i}$ which also has boundary $L(n_i^2,n_im_i-1)$ (for more details see Sect. \ref{constr}). 
To simplify the notation, we denote by $B_{n,m}$ the rational homology ball with boundary $L(n,nm-1)$
used in the generalized rational blowdown surgery.
On these spaces, the author has proved the existence of preferred metrics:

\begin{thm}\cite{ricflat}
\label{met-rhb}
The rational homology ball $B_{n,m}$ admits an asymptotically locally Euclidean (ALE) Ricci flat K\"ahler metric $g_{n,m}.$
Moreover, any ALE Ricci flat K\"ahler surface $(X,g)$ with Euler characteristic $\chi(X)=1$ is isomorphic, up to rescaling, to 
$(B_{n,m}, g_{n,m}),$ for some integers $0<m<n, \gcd(n,m)=1.$ 
\end{thm}

To prove this theorem it was essential to require that metric is asymptotically locally Euclidean 
(see Sect. \ref{constr} for the exact definition). 
This property allows us to construct a one point conformal compactification of $B_{n,m}$, 
which has positive orbifold Yamabe invariant, which in turn is used to prove Theorem \ref{main}.
More precisely, we prove the following proposition which is of independent interest:
\begin{propo}
\label{yam-rat-ball}
The rational homology ball $(B_{n,m}, g_{n,m})$ admits a one-point conformal compactification   
$(\wh B_{n,m},\wh g_{n,m})$ with positive orbifold conformal Yamabe
$$Y^{orb}(\wh B_{n,m},[\wh g_{n,m}])=Y(S^4)/n=\frac{8\pi\sqrt6}{n}.$$
Moreover, there is no solution to the orbifold Yamabe problem on $(\wh B_{n.m}, [\wh g_{n,m}]),$ 
i.e. there is no conformal metric $\wt g=e^f \wh g_{n,m}, f\in C^\infty(\wh B_{n,m})$ having constant scalar curvature.
\end{propo}

This proposition is an extension of Viaclovsky's result on ALE hyperk\"ahler 4-manifolds \cite[Theorem 1.3]{Via}.

\medskip

Behind most of the results in this paper is the Seiberg-Witten theory. 
This is essential in the computation of the invariant in Theorem \ref{Yam-Kah} and the curvature bound \ref{Y-bound}. 
For Propositions \ref{prop2} and \ref{prop22}, where Seiberg-Witten theory can be used in the proofs, we give alternative, topological or differential geometric arguments. 

At the time when this paper was written, Su\'arez-Serrato and Torres  \cite{tor} have posted a result on the computation of the
Yamabe invariant for symplectic $4-$manifolds of Kodaira dimension $0$ and $1.$
\section{The generalized rational blowdown surgery}
\label{constr}

The (generalized) rational blowdown is a surgery construction specific to dimension $4$ 
introduced by Fintushel and Stern \cite{FiSt}.
 It is the analog of the  blow-down construction in algebraic geometry, 
 and it is motivated by the study of the local structure of cyclic quotient singularities.
\medskip

In complex dimension two, quotient singularities of finite cyclic groups and their minimal resolution are well understood. 
We denote by  $\CC^2/\frac1p(1,q)$ the quotient of $\CC^2$ by the finite cyclic group of order $p$ 
acting diagonally with weights $1$ and $q,$ for $q<p$  relatively prime integers.
We say that we have a singularity of type $\frac 1p(1,q).$ 
Its minimal resolution consists of a connected chain of transversal rational curves 
$E_1,\cdots, E_k$, with the self-intersection $E_i^2=-e_i, e_i\geq2$ given 
by the Hirzebruch-Jung continued fraction 
\begin{equation}
  \frac pq = e_1 - \cfrac{1}{e_2-
         \cdots
           \cfrac{1}{e_k } } .
\end{equation}
 A neighborhood of infinity is of the form $L(p,q)\times (T,\infty),$ where $L(p,q)$ denotes the lens space 
 associated to $\CC^2/\frac1p(1,q)$ and $T\gg0$. 
 In the case when the singularity is of type $\frac 1{n^2}(1,nm-1),$ with $n, m$ relatively prime integers, 
 Koll\'ar and Shepherd-Barron \cite{ksb} exhibit a one-parameter smoothing 
for which the associated Milnor fiber is a rational homology ball $B_{n,m}.$ 
 The induced complex structure $J$ on $B_{n,m}$ has the property that the n$^{th}$ tensor product
 of the canonical line bundle, $K_{B_{n,m}}^n,$ is trivial.
 \medskip

Let $M$ be a $4-$manifold containing a linear chain $C_{n,m}$ of embedded $2-$spheres  
$E_1, \cdots, E_k,$ which have negative self-intersections $(-e_1,\cdots,-e_k),$ where $e_i$ are 
obtained by the continued fraction expansion of  $\frac {n^2}{nm-1}.$ 
We construct a manifold $N$ by removing a small tubular neighborhood of the configuration of spheres, 
and gluing in the rational homology ball $B_{n,m},$ 
by identifying the two neighborhoods of infinity which are of the form 
$L(n^2, nm-1)\times (T,T')\subset L(n^2, nm-1)\times (T,\infty).$
 This surgery was introduced by Fintushel and Stern \cite{FiSt}  in the  case when $m=1,$
 and it is called the rational blowdown.
 In general, the construction is known as the generalized rational blowdown. 
 Symington  \cite{sym1,sym2} proved that when $M$ is a symplectic manifold
 and the chain $C_{n,m}$ is formed by symplectic submanifolds, than the new manifold $N$ admits a symplectic structure. 
 Throughout this paper, we do not  emphasize a particular choice of a symplectic structure. 
 The existence of one such structure is sufficient to prove the curvature estimate (\ref{CL-bound}).

\medskip

One can also consider K\"ahler surfaces $M$ containing  configurations of complex submanifolds 
given by  linear chains of rational curves in one of the above arrangements. 
A third manifold which is associated in a natural way to this construction is obtained by 
collapsing the chains of curves to singular points. 
We denote this by $\wh M.$
Artin's Contractibility Criterion tells us that $\wh M$ has a complex structure, and as we used preferred configurations,
the singularities are of type $\frac 1{n^2}(1,nm-1).$
Locally, the generalized rational blowdown can be described as the $\QQ-$Gorenstein smoothing of the quotient singularity \cite{man}.
In general, it is difficult to determine if the associated manifold $N$ is K\"ahler or not. 
In particular cases, as the examples in \cite{lp,rs1}, one can show that a global $\QQ-$Gorenstein smoothing 
of  $\wh M$ exists and conclude that the manifold $N$ is in fact K\"ahler. 
In Sect. \ref{non-K-exp} we give a family of examples for both cases, when the surgery yields K\"ahler, or non-K\"ahler manifolds.

\medskip

Next, we emphasize some of the relations between the properties of the three manifolds. 
Propositions \ref{prop1}, \ref{prop2} and \ref{prop22} are stated for  $M$ and $\wh M$  K\"ahler surfaces as this is sufficient for our purposes. 
Similar statements hold for symplectic $4-$ manifolds, with almost identical proofs.

\begin{prop}
\label{prop1}
 Let $M$ be a K\"ahler surface containing $k$ linear chain configurations of negative self-intersection rational curves $C_i, i=1,\cdots,s$, associated to singularities of type $\frac1{n_i^2}(1, n_im_i-1),$ with $n_i>m_i$ and relatively prime.
Let $N$ denote the symplectic manifold obtained from $M$ by taking the generalized rational blowdown of these configurations of curves.
Let $\wh M$ be the complex orbifold surface obtained by collapsing these configurations to orbifold points.
Then:
\begin{itemize}
\item[(1)]
\begin{equation}
c_1^2(\wh M)=c_1^2(N)
\end{equation}
\item[(2)] If all the $m_i=1$ then 
\begin{eqnarray}
c_1^2(N)=c_1^2(M)+(\sum_{i=1}^{s} (n_i-1)), \notag\\
c_2(N)=c_2(M)-(\sum_{i=1}^{s} (n_i-1))
\end{eqnarray}
\end{itemize}
\end{prop}

\begin{proof}

(1) The square of the first Chern class can be computed in terms of the orbifold Euler characteristic and the orbifold signature:
\begin{equation*}
c_1^2(\wh M)=2\chi^{orb}(\wh M)+3\tau^{orb}(\wh M).
\end{equation*}

Moreover, the orbifold Euler characteristic and signature are related to the topological ones as follows:
\begin{equation*}
\chi^{orb}(\wh M)=\chi^{top}(\wh M)-\sum_{i=1}^{s}\frac{|\Gamma_i|-1}{|\Gamma_i|}
\end{equation*}
and
\begin{equation*}
\tau^{orb}(\wh M)=\tau^{top}(\wh M)+\sum_{i=1}^{s}\eta (S^3/\Gamma_i)
\end{equation*}
where $\Gamma_i$ is the local group associated to the singularity, and
 $\eta(S^3/\Gamma_i)$ is the eta-invariant of the link of the quotient singularity. 
In our case $|\Gamma_i|$ is $n_i^2,$ and the eta-invariant of the singularity is computed by Park \cite[Lemma 3.2]{park} to be:
\begin{equation*}
\eta(S^3/\frac1{n^2}(1,nm-1))=\frac23(1-\frac1{n^2})
\end{equation*}
Hence: 
\begin{equation*}
c_1^2(\wh M)=2\chi^{orb}(\wh M)+3\tau^{orb}(\wh M)=2\chi^{top}(\wh M)+3\tau^{top}(\wh M)
\end{equation*}
As $N$ is obtained from $\wh M$ by removing the singular points and gluing in rational homology spheres,
the Mayers-Vietoris sequence and Novikov's Additivity Theorem imply that 
\begin{equation*}
\chi^{top}(\wh M)=\chi(N), \tau^{top}(\wh M)=\tau(N).
\end{equation*}
We have immediately:
\begin{equation*}
2\chi^{top}(\wh M)+3\tau^{top}(\wh M)=2\chi(N)+3\tau(N)=c_1^2(N).
\end{equation*}

(2) In the case when $m_i=1,$ we only consider quotients of the form $\frac1{n_i^2}(1,n_i-1)$ for which 
the minimal resolution is a chain of  $n_i-1$ spheres of self-intersection  $(-(n_i+2),-2,\cdots,-2).$ 
Hence, the rational blow-down of such a configuration  increases $c^2_1=2\chi+3\tau$ by $n_i-1$ 
and decreases $c_2=\chi$ by $n_i-1,$ for
each configuration.
\end{proof}

Recall that a symplectic $4-$manifold is called minimal if it does not contain symplectically embedded  $2-$spheres 
of self-intersection $(-1).$ 
\begin{defn}[\cite{no-ein,tjl}]
A minimal symplectic $4-$manifold $(M,\omega)$ is called of general type if it satisfies:
\begin{equation}
c_1^2(M)>0 {\text ~and ~}
c_1(M)\cdot [\omega]<0,
\end{equation}
where $c_1(M)$ is the first Chern class of an almost complex structure compatible with $\omega.$
Moreover, a symplectic $4-$manifold is called of general type if it is the iterated blow-up of a minimal symplectic manifold of general type.
\end{defn}

This definition extends the classical one from the complex surface case. 
The first condition in the definition is  topological. 
The second condition only depends on the underlying smooth structure and
 it is independent of the choice of the symplectic structure \cite[Theorem 2.4]{tjl}.

\begin{prop}
\label{prop2}
 With the above notations, if $(\wh M, \omega_{\wh M})$ is an orbifold K\"ahler surface such that 
 $c_1^2(\wh M)>0$ and $c_1(\wh M)\cdot[ \omega_{\wh M}]<0$  then $N$ is a symplectic manifold of general type.
\end{prop}

\begin{proof}
As $c_1^2(\wh M)>0,$ then  $c_1^2(N)=c_1^2(\wh M)>0,$ and we only need to show
 that $c_1(N)\cdot [\omega_N]<0,$ or equivalently $c_1(K_N)\cdot[\omega_N]>0,$ for some choice of $\omega_N.$

In order to simplify the notations, we show that $c_1(N)\cdot[\omega_N]<0$ in the case when $M$ contains 
only one configuration of curves of type $C_{n,m}.$ 
The computations extend trivially to $s$ configurations of this type. 

Let  $M_0$ be the open manifold such that:
\begin {eqnarray}
M&=&M_0\bigcup_{L(n^2,nm-1)} Nbd (C_{n,m}), \notag\\
\wh M&=&M_0\bigcup_{L(n^2,nm-1)} \CC^2/\frac1{n^2}(1,nm-1), \notag\\
 N&=&M_0\bigcup_{L(n^2,nm-1)} B_{n,m}.
\end{eqnarray}
On $M$ we consider the K\"ahler structure obtained from the pull-back of the K\"ahler structure on 
$\wh M$ and adding a sufficiently small multiple 
of a K\"ahler structure on the  $Nbd(C_{n,m})$ which is diffeomorphic to the minimal resolution of $\CC^2/\frac1{n^2}(1, nm-1).$
Let $c_1(K_N)$ and $\omega_N$ be the canonical class and the symplectic form on $N,$ constructed by Symington \cite{sym1, sym2}. 
Moreover, Symington shows that there is a symplectomorphism $\psi :(M_0, \omega_N|M_0)\to (M_0,\omega_{\wh M}|M_0).$ 
Since $H^1(L(n^2,nm-1),\QQ)=H^2(L(n^2,nm-1),\QQ)=0,$ we have that $c_1(K_N)$ and $\omega_N$ decompose as 
$$c_1(K_N)=c_1(K_N)|_{M_0}+c_1(K_N)|_{B_{n,m}} ~\text{and}~ [\omega_N]=[\omega_N|_{M_0}]+[\omega_N|_{B_{n,m}}]$$
and as $(K_M|_{B_{n,m}})^n$ is trivial, then:
\begin{eqnarray}
c_1(K_N)\cdot[\omega_N]&=&{c_1(K_N)|}_{M_0}\cdot[\omega_N|_{M_0}]=\psi_*(c_1(K)|_{M_0})\cdot\psi_*([\omega_N|_{M_0}])\notag\\
&=&c_1(K_{\wh M})\cdot[\omega_{\wh M}]>0.\notag
\end{eqnarray}
\end{proof}

The structure of $N$ is in fact intimately related to the properties of $\wh M,$ and not necessarily to those of $M.$

\begin{prop}
\label{prop22}
Consider the complex orbifold surface $\wh M'$ obtained from $\wh M$ by collapsing all the rational curves of self-intersection $(-2)$ which do not pass through the singular points of $\wh M.$ 
If the canonical sheaf of $\wh M'$ is ample  then $N$ is a minimal symplectic surface of general type.
\end{prop}

\begin{proof} Since $\wh M'$ has ample canonical sheaf, we have $c_1^2(\wh M')>0.$
Moreover, by the positive solution to the Calabi Conjecture for orbifolds, due to Kobayashi \cite{kob}, 
there exists a K\"ahler-Einstien metric $g'$ such that $ \rho_{g'}=\lambda\omega_{g'}$, 
where $\rho_{g'}$ is the Ricci form associated to $g'$ and $\lambda<0$ is the Einstein constant. 
Then  we have 
$$c_1(\wh M')[\omega_{g'}]=[\frac1{2\pi}\rho_{g'}][\omega_{g'}]=\frac1{2\pi}\lambda[\omega_{g'}]^2<0.$$
The projection $p:\wh M\to \wh M'$  collapses only self-intersection $(-2)$ rational curves, 
and hence $K_{\wh M}=p^*(K_{\wh M'}).$ 
This implies that $c_1(\wh M)^2=c_1(K_{\wh M})^2=c_1(K_{\wh M'})^2>0.$ 
We can also construct a K\"ahler form on $\wh M$ by taking the pull-back of the K\"ahler form $\omega_{g'}$ and adding a small multiple of a K\"ahler form $\omega_0$ on $\wh M:$
$$\omega:= p^*(\omega_{g'})+\epsilon \omega_0, \epsilon>0 ~\text{sufficiently small}.$$
Then 
$$c_1(\wh M)[\omega]=[\frac1{2\pi}p^*(\rho_{g'})][p^*(\omega_{g'})+\epsilon \omega_0]=\frac1{2\pi} \lambda[\omega_{g'}]^2+\epsilon c_1(\wh M)\omega_0<0$$
for $\epsilon$ sufficiently small.
Then, by Proposition \ref{prop2}, the manifold $N$ is of general type.

One method to show that $N$ is a minimal surface is by using Seiberg-Witten theory 
to show that $N$ has a unique basic class up to sign, see \cite{park}. 
We  omit this proof here, and instead we give a different proof in Sect. \ref{bound}, using curvature estimates.

\end{proof}

\section{The Yamabe invariant and curvature estimates}
\label{bound}

Let $M$ be a closed $4-$manifold.
We first define the Yamabe constant, $Y(M, [g]),$ as an invariant of the conformal class of a Riemannian metric $g$ on $M:$ 

$$Y(M,[g])=\inf_{\wt g\in[g]}\frac{\int_M s_{\wt g} d\mu_{\wt g}}{{Vol(\wt g)}^{\frac12}},$$
where $[g]=\{\wt g=e^fg| f:M\to\RR \text{ smooth}\}.$

The Yamabe invariant is a diffeomorphism invariant of the manifold and is defined  \cite{oKob,schoen} as:

\begin{equation}
\label{yam}
Y(M)=\sup_{[g]} Y(M,[g]).
\end{equation}

We use  an alternative description in the non-positive case 
which  helps us to compute the invariant. 

\begin{prop}
\label{comp-yam} Let $M$ be a smooth compact $4-$manifold.
\begin{itemize}
\item[(1)]   \cite{bcg} Let $[g]$ be a conformal class with negative
Yamabe constant. Then
\begin{equation}
\inf_{\wt g\in[g]}\int_M s_{\wt g}^2 ~ d \mu_{\wt g}=|Y(M,[g])|^2,
\end{equation}
and this infimum is precisely achieved by the metrics of constant scalar curvature.
\item[(2)] \cite{oKob,schoen}
If  $Y(M)\leq0$ then
\begin{equation}
\inf_{g}\int_M s_g^2 ~ d \mu_g=|Y(M)|^2,
\end{equation}
where the infimum is taken over all smooth Riemanian metrics $g$ on $M.$ 
\end{itemize}
\end{prop}

Using this description, LeBrun was able to compute the Yamabe invariant in the case of K\"ahler surfaces 
of non-negative Kodaira dimension. 
He shows that:

\begin{thm}[\cite{no-ein}, \cite{kod-dim}]
\label{Yam-Kah}
Let $M$ be obtained from a minimal K\"ahler surface $X$ of  Kodaira dimension $0, 1$ or $2$ by blowing up $l\geq 0$ points. 
Then the Yamabe invariant  $Y(M)$ is independent of $l,$ and
\begin{equation}
Y(M)=-4\pi\sqrt{2c_1^2(X)}.\notag
\end{equation}
Moreover, this invariant is not attained if $l>0.$
\end{thm}

\subsection{The orbifold Yamabe invariant}
Orbifold spaces  appear naturally in the study of the moduli space of metrics on smooth manifolds \cite{TiVi, And}. 
We  only consider spaces with isolated singular points, where there exists a neighborhood $(U,p)$ 
of the singular point $p$ modeled by 
$$\psi: (U,p)\to (B_{r_0}(0)/\Gamma_p,0),$$
where $B_{r_0}$ is a ball of radius $r_0$ in $\RR^4$ and $\Gamma_p$ is a finite subgroup of $O(4)$ 
acting freely on $\RR^4\setminus\{0\}.$ 
An orbifold metric $g$ on $M$ is a Riemannian metric on the smooth part of $M$ and such that the pull-back of 
$\psi_*(g)$ on $B_{r_0}(0)\setminus\{0\}$ extends to a smooth Riemannian metric on $B_{r_0}(0).$ 
We denote by $\cal{M}^{orb}(M)$ the space of orbifold metrics on $M$, and by $[g]_{orb}=[g]\cap \cal{M}^{orb}(M)$ 
the orbifold conformal class of $g.$ 
The orbifold Yamabe constant of a conformal class and the orbifold Yamabe invariant of $M$ can be defined 
in an analogous way to the smooth case (see  \cite{AkBo} for detailed definitions and properties).

\medskip
Let $B_{n,m}$ be one of the rational homology balls discussed in Sect \ref{constr}.
In \cite[Proposition 3.3]{ricflat}, the author shows that $B_{n,m}$ admits an ALE Ricci flat K\"ahler metric $g_{n,m}$. 
The ALE requirement  means that there exists a diffeomorphism, 
$$\phi:U_{\infty}\to(\CC^2\setminus B_T(0))/\frac{1}{n^2}(1,nm-1), T\gg0$$
 between a neighborhood of infinity in $B_{n,m}$ and a neighborhood of infinity in 
 $\CC^2_{(z_1,z_2)}/\frac{1}{n^2}(1,nm-1),$ 
 such that $\phi_*(g_{n,m})=g_{Eucl}+O(R^{-4}),$
 where $R$ denotes the Euclidean radius.
 Then if we consider the change of coordinates $x_i=z_i/R^2, i=1,2,$  we see that the rational homology ball $(B_{n,m},g_{n,m})$ 
 admits a  $C^{1,\alpha}$ conformal compactification 
 with one orbifold point of type $\CC^2/\frac1{n^2}(1,nm-1).$ 
 We denote by $\wh B_{n,m}$ the compactification.

Next, we prove that the results in the hyperk\"ahler $4-$case \cite[Theorem 1.3]{Via} extend to the orbifold compactification of 
 the rational homology ball $B_{n,m}.$

\begin{proof}[Proof of Proposition \ref{yam-rat-ball}]
Notice that the existence of a $C^\infty$-orbifold conformal compactification $(\wh B_{n,m},[\wh g_{n,m}])$  follows 
 directly from \cite[Proposition 12]{clw}. 
 Moreover, the arguments in the proof of \cite[Proposition 13]{clw} tell us that we can choose a metric 
 $g_\infty$ in the conformal class of $\wh g_{n,m}$ which has strictly positive scalar curvature everywhere.  
  We can now conclude that  $Y(\wh B_{n,m},[\wh g_{n,m}])>0.$
 We remark that in order to apply Propositions 12 and 13 from \cite{clw} we only need  ALE scalar flat anti-self-dual metrics, 
  which is a condition much weaker than the ALE Ricci-flat K\"ahler condition.

 On the other hand, \cite[Theorem B]{AkBo} tells us that the orbifold Yamabe invariant is bounded from above in terms of the
 Yamabe invariant of the $4-$sphere and the order of the group of the singularity, $|\Gamma|=n^2$:
 $$Y^{orb}(\wh B_{n,m}) \leq \frac{Y(S^4)}{\sqrt{|\Gamma|}}=\frac{8\sqrt6\pi}{n}.$$
Moreover, Akutagawa \cite[Theorem 3.1]{aku} shows that if the bound is strict than there exists a 
metric $\wt g\in[\wh g_{n,m}]$ which has constant scalar curvature. 
On the other hand $B_{n,m}$ has fundamental group $\ZZ/n\ZZ,$ and its universal covering is the manifold $A_{n-1},$ 
diffeomorphic to the minimal resolution of $\CC^2/\frac1n(1, n-1).$ 
Given the covering map $c:A_{n-1}\to B_{n,m},$ then $(A_{n-1}, c^*(g_{n,m})$ 
is a hyperk\"ahler $4-$manifold.
Hence, the existence of a metric of constant scalar curvature $\wt g\in[\wh g_{n,m}]$ implies the existence of a 
constant scalar curvature metric in the orbifold conformal compactification class of $c^*(g_{n,m}).$ 
But Viaclovsky shows that this is impossible \cite[Section 4.1]{Via}. 
Hence the orbifold Yamabe invariant is the maximal one, and there is no metric of constant scalar curvature in the given conformal class.
\end{proof}

\subsection{Curvature estimates}
There are well-know relations between the $L^2-$norms of the curvature components and 
the topological invariants of $4-$manifolds given by the Gauss-Bonnet theorem, or the Hirzebruch Signature theorem. 
As we consider manifolds with special geometries, K\"ahler or symplectic manifolds, 
 there are additional curvature estimates which depend on the differentiable structure (see for example \cite{no-ein}). 
In this paper we use the following key scalar curvature bound, which relies on the existence of a nontrivial Seiberg-Witten invariant:

\begin{thm}[{\cite[Theorem 4]{no-ein}}]
\label{CL-bound}
Let $(X,\omega)$ be a symplectic manifold, and let $M=X\#l\cpb.$ If $b^+=1,$ assume that 
$c_1(X)\cdot[\omega]<0.$
Then any Riemannian metric on $M$ satisfies
\begin{equation*}
\int_Ms^2d\mu\geq32\pi^2 c_1^2(X).
\end{equation*}
\end{thm}

This bound is  sharp in the case of minimal complex surfaces of 
Kodaira dimension $0,1,$ or $2$  \cite{kod-dim,no-ein}. 
In general, the class of symplectic manifolds is much larger than that of K\"ahler surfaces, 
and we do not expect  that the bound is optimal for an arbitrary manifold.
However, we show that this bound is also sharp for the class  of symplectic manifolds constructed in Sect. \ref{constr}.

\begin{thm}
\label{metrics}
Let $M$ be a compact complex surface containing $s$ configurations of rational curves of type $C_{n_i,m_i}$, with $m_i<n_i$ and $\gcd(n_i,m_i)=1.$
Let $\wh M'$ be the orbifold surface obtained by collapsing these chains and all the rational curves of self-intersection $(-2)$ which are
disjoint from these configurations. 
Assume that the canonical sheaf of $\wh M'$ is ample. 
Let $N$ be the manifold obtained by taking the generalized rational blowdown of the configurations $C_{n_i,m_i}$ in $M.$
Then there exists a family of metrics $g_t$ on $N$ such that 
\begin{equation}
 \lim_{t\to0^+}\int_N s_{g_t}^2 d\mu_{{g_t}}=32\pi^2c_1^2(N).
 \end{equation}
Moreoever, if $N'=N\#l\cpb,$ then there exists a family $g'_t$ of metrics on $N'$ such that  
\begin{equation}
\label{mtr-bl}
 \lim_{t\to0^+}\int_{N'} s_{g'_t}^2 d\mu_{{g'_t}}=32\pi^2c_1^2(N).
 \end{equation}
\end{thm}

\begin{proof}
The singular points of $\wh M'$ are isolated, of either type $\frac1{n^2}(1,nm-1),$ 
when a configuration of type $C_{n,m}$ is collapsed, 
or rational double points, in the case when a connected configuration of $(-2)$ curves is contracted. 
The singularities of the last type are of the form $\CC^2/G,$ with $G\subset SU(2)$ finite subgroup, 
acting freely on $\CC^2\setminus\{0\},$ and are  called $ADE$ singularities, in correspondence with the
Dynkin diagram associated to the configuration of curves appearing in the minimal resolution.
We refer to the minimal resolution as the $ADE-$space.
Then the manifold $N$ is obtained from the orbifold surface $\wh M'$ by removing the singular points and
gluing in  rational homology balls or $ADE-$spaces, respectively.

The proof of the theorem follows the line of arguments in \cite{no-ein}, 
where a family of metrics saturating the bound in Theorem \ref{CL-bound} is constructed.
We give the explicit construction of these metrics.

Since the orbifold surface $\wh M'$ has ample canonical sheaf, then 
 the solution of the orbifold version of the Calabi Conjecture \cite{kob} 
endows $\wh M'$ with an orbifold K\"ahler-Einstein metric $\hat g.$ 
The scalar curvature $s_{\hat g}$ and Ricci form $\hat \rho$ of $\hat g$  satisfy:
$$\int_{\wh M'}s_{\hat g}^2 ~d\mu_{\hat g}=8\int_{\wh M'} {\hat \rho}^{~2}=32\pi^2 c_1^2(\wh M').$$
Consider $\hat g-$geodesic coordinates around each orbifold point of $\wh M'$ 
such that the metric $\wh g=g_{Eucl}+O(r^2),$ where $g_{Eucl}$ and $r$ denote the Euclidean metric
and the radius associated with the chart, respectively. 
Let $h_1=g_{Eucl}-\wh g\in O(r^2).$

On the complement of the open neighborhood of infinity $U_\infty\subset B_{n,m}$ we decompose the metric 
$g_{n,m}$ as $g_{Eucl}+h_2,$ with $h_2\in O(R^{-4}).$ 
We change the coordinates of 
$U_\infty\cong(\CC^2/\frac1{n^2}(1, nm-1)\setminus B_T(0))$  such that 
$U_\infty\cong(\CC^2/\frac1{n^2}(1, nm-1)\setminus B_ {\frac t2}(0)),$ and then rescale the metric. 
Then, on $B_{n,m}$ we have a family of charts at infinity defined for $r>\frac t2$ such that the ALE Ricci-flat K\"ahler metric is 
asymptotically  of the form $g_{Eucl}+h_2',$ with $h_2'\in O(t^2).$ 
The volume of the compact set $B_{n,m}\setminus U_\infty$ is rescaled by a factor of $(\frac t {2T})^4.$
The $ADE-$spaces have hyperk\"ahler metrics $g_{hyp},$ with the same asymptotic description \cite{kron1}. 
Hence, we can introduce a common notation for a neighborhood of infinity, and the metric decomposition in both cases.
For the rest of the proof $(U_\infty, g_\infty)$ denotes a neighborhood of infinity in $B_{n,m}$ or in an $ADE-$space, which is isometric
to $(\CC^2/G\setminus B_{\frac t2}(0), g_{Eucl}+h_2'),$ where $G$ is $\frac1{n^2}(1,nm-1)\subset U(2)$ or $G\subset SU(2),$ respectively.

In order to define a metric on $N,$ we remove a ball of radius $t$  around each singularity in $\wh M'$ and 
glue in a rational homology ball $B_{n,m}$ or an $ADE$-space. 
The new metric  interpolates 
between the metric $\wh g$ and the rescaling of the metric $g_{n,m}$ or $g_{hyp}$ in the region $t<r<2t,$ for some small parameter $t.$
As the gluing is done by identifying the region $\{z\in \CC^2/G|~~  t<|z|<2t\}\subset U_\infty$ 
to a region given by the same local coordinates in a neighborhood of a singular point in $\wh M',$ 
we can do the gluing on the covering $\{z\in\CC^2|~~  t<|z|<2t\}.$

Let $\phi:[0,4]\to \RR$ be a non-negative smooth function such that $\phi$ is identically $1$ on the interval $[0,1]$ and identically $0$
on the interval $[2,4].$ 
We define the metric $g_t$ on the larger collar $\{z\in\CC^2|~~ \frac t2<|z|<4t\}$ :
$$g_t:=\wh g+ \phi(\frac{r}t)(h_1+h_2).$$
Then for $r>2t$ the metric $g_t$ is identical to $\wh g,$ while for $\frac t2<r<t$ the metric is the rescaled metric $g_\infty.$
Hence, we have defined a family of metrics $(N, g_t).$

On the interpolation region $t<r<2t$ we have the following uniform bounds:
\begin{align*}
||g_t- g_{Eucl}||&\leq C t^2  \\
||D g_t||&\leq Ct \\
||D^2 g_t||&\leq C, 
\end{align*}
where $D$ is the Euclidean derivative operator associated with the given coordinate system, and the constant $C$ is independent of $t.$
Then $s^2_{g_t}$ is uniformly bounded as $t\to 0.$ Since the volume of the annular transition region goes to zero, we have:
$$\lim _{t\to 0^+} \int_N s_{g_t}^2 d\mu_{g_t}=\int _{\wh M} s_{\wh g}^2 d\mu_{\wh g}=32\pi^2 c_1^2(\wh M)=32\pi^2c_1^2(N).$$

The blow-up  manifold, $N'=N\#l\cpb,$ can be obtained from $N$ by removing $l$ 
smooth points and gluing in the total space of $\OO_{\bcp^1}(-1),$ for each of the points. 
The total space of $\OO_{\bcp^1}(-1)$ admits an ALE scalar flat K\"ahler metric called the Burns metric. 
The  family $g_t'$ satisfying equation \ref{mtr-bl} can be obtained from the family of metrics $g_t$ 
using the explicit gluing in \cite[Sect. 5]{no-ein}.
\end{proof}

\begin{rmk}
There is no a priori reason for the family of Riemannian metrics  in the above construction 
to be compatible with an underlying symplectic structure. 
\end{rmk}

Then, we immediately have that:
\begin{cor} 
\label{invar}
For  $N$ and $N'$ given above we have:
\begin{itemize}
\item[(1)] $$\inf_g \int_N s_g^2 ~d\mu_g=\inf_g \int_{N'} s_g^2 ~d\mu_g =32\pi^2c_1^2(N),$$
\item[(2)] the Yamabe invariant $Y(N)=Y(N')=-4\sqrt{2c_1^2(N)}.$
\end{itemize}
\end{cor}

\begin{proof}
Part $(1)$ is a consequence of the above proposition and the curvature bound in Theorem \ref{CL-bound}. Part $(2)$  follows from Proposition \ref{comp-yam}, and the fact that the symplectic manifolds of general type have negative Yamabe invariants \cite{no-ein}.
\end{proof}

We can now give a proof for the second part of Proposition \ref{prop22}:
\begin{proof}[Proof of minimality in Proposition \ref{prop22}]
\label{minimality}
Assume that the manifold $N$ is obtained by blowing up the symplectic manifold $X$ at $l'>0$ points, $N=X\#l'\cpb.$ 
Then  $c_1^2(N)=c_1^2(X)-l',$ and by Theorem \ref{CL-bound}:
$$\int_N s_g^2 d\mu_g\geq 32\pi^2c_1^2(X)=32\pi^2(c_1^2(N)+l'), ~l'>0$$
 for any metric $g$ on $N.$ But this is in contradiction with Corollary \ref{invar}.
\end{proof}

\begin{proof}[Proof of Theorem \ref{main}]
 Proposition \ref{prop22} implies that $N$ is a minimal symplectic manifold of general type. 
Moreover,  Corollary \ref{invar} and Proposition \ref{prop1} tell us that:
$$Y(N)=Y(N')=-4\pi\sqrt{2c_1^2(N)}=-4\pi\sqrt{2c_1^2(\wh M)}.$$
Moreover, as $p^*(K_{\wh M'})=K_{\wh M},$ the above equality also equals $-4\pi\sqrt{c_1^2(\wh M')}.$
We only need to show that the orbifold Yamabe invariants of $\wh M$ and $\wh M'$ are the expected ones.
In order to do this, we generalize Theorem B in \cite{aku} to include orbifolds which have 
singularities of type $\frac1{n^2}(1,nm-1).$ 

We can give a second description of the manifold $N$ 
as the connected sum of the orbifold spaces $\wh M$ and $\wh B_{n_i,m_i}, i=1,\dots, s$ at the singular points. 
Similarly, $\wh M$ can be obtained by taking the connected sum of $\wh M'$ and the
one-point compactifications of $ADE-$spaces.

In \cite{aku}, Akutagawa considers the connected sum  of two orbifold $4-$spaces $M_1$ and $M_2$
at two isolated orbifold points of type $\RR^4/\Gamma$:
$$N:=M_1\#_{S^3/\Gamma}M_2.$$ 
He proves that if  $Y(N)<0$ and $Y^{orb}(M_2)>0$ then $Y^{orb}(M_1)\leq Y(N)\leq0.$
 Proposition \ref{yam-rat-ball} tells us that the one-point compactification  $(\wh B_{n,m},[\wh g_{n,m}])$ 
 has positive Yamabe invariant. 
Applying Theorem A in \cite{aku} for $M_1=\wh M$ and $M_2=\wh B_{n,m}$ we see that $Y^{orb}(\wh M)\leq Y(N)<0.$ 

The one point conformal compactification of an $ADE-$space has positive Yamabe invariant  \cite[Theorem $1.3$]{Via},
and similarly,  we have $Y^{orb}(\wh M')\leq Y^{orb}(\wh M).$

On the other hand,  $\wh M'$ admits a K\"ahler-Einstein metric $\wh g$ given by the positive solution of the Calabi conjecture \cite{kob}. 
This is the Yamabe minimizer in its conformal class, and satisfies:

$$\displaystyle \frac{\int_{\wh M'} s_{\wh g}d\mu_{\wh g}}{Vol^{\frac12}({\wh g})}=s_{\wh g}=-\sqrt{32\pi^2c_1^2(\wh M')}=-\sqrt{32\pi^2c_1^2(N)}=Y(N).$$ 
Then  we have that $Y^{orb}(\wh M')=Y(N)=-4\pi\sqrt{2c_1^2(\wh M')},$
and so $Y^{orb}(\wh M')=Y^{orb}(\wh M).$
\end{proof}

\section{An explicit family of symplectic, non-K\"ahler examples} 
\label{non-K-exp}


The family of examples described in this section was first introduced by Fintushel and Stern \cite{FiSt}, 
and later reconsidered by Y. Lee and J. Park \cite{lp2}.
In this paper, we follow the second description. 

\medskip

Let $E(n)$ be the relatively minimal, simply connected elliptic surfaces with sections, 
which have Euler characteristic $\chi(E(n))=12n$. 
The diffeomorphism type of $E(n)$ is unique \cite[Theorem 3.2.9]{GoSt}, but the complex structure might vary. 
For $n\geq4,$ Lee and Park \cite{lp2} exhibit such a preferred complex structure with two sections of self-intersection $(-n),$ 
which are part of two disjoint configurations of $(n-3)$ curves of type $(-n,-2,\cdots,-2).$ 
Let $\wh X_n$ be the manifold obtained by contracting one chain of type $(-n,-2,\cdots,-2)$ 
to a singular point $P$ of type $\frac1{(n-2)^2}(1,n-3),$ and the $(n-4)$ spheres of self-intersection $(-2)$ 
from the second chain to a singular point $Q$ of type $\frac1{n-3}(1,n-4).$
We have:

\begin{prop}
\label{orb-ampl}
For every $n\geq4,$ there exists an orbifold surface $\wh X_n$ with two orbifold points of type 
$\frac1{(n-2)^2}(1,n-3)$ and $\frac1{n-3}(1,n-4),$ such that its canonical sheaf is ample and $c_1^2(\wh X_n)=n-3. $ 
\end{prop}

\begin{proof}
The first  case, when $n=4$ appeared earlier in the work of Gompf \cite{Go}, 
and the surface $\wh X_4$ has only one orbifold point, which is of type $\frac14(1,1).$  
For clarity, we treat this case separately.

Let $\Sigma(4)=\PP(\OO_{\bcp^1}\oplus\OO_{\bcp^1}(4))$ be the Hirzebruch surface with a negative section of self-intersection $(-4).$ 
We denote by $C_0$ the negative self-intersection section, and by $f$ a generic fiber. 
Let $D$ be a smooth complex curve in the linear system $|4(C_0+4f)|.$ 
The double cover of $\Sigma(4)$ branched along $D$ is the elliptic surface $E(4).$ 
As the branch locus $D$ is disjoint from $C_0,$ the pullback of $C_0$ has two connected components,
which are both sections of the elliptic fibration and have self-intersection $(-4).$  
We  denote them by $C_1$ and $C_2.$ 
Gompf shows that if we consider the rational blowdown of both $C_1$ and $C_2$ the new manifold admits a complex structure, 
and it is a Horikawa surface \cite{bpv}, while if we consider the rational blowdown of only one curve the new symplectic $4-$manifold does not admit a complex structure.
Let $\wh X_4$ be the orbifold surface obtained by collapsing  only $C_2.$ 
Following the description of $E(4)$ as a double cover $d: E(4)\to\Sigma(4),$ we can compute the canonical divisor \cite[Lemma 17.1]{bpv}:
\begin{eqnarray*}
K_{\Sigma(4)}&=&-2C_0-6f\\
K_{E(4)}&=&d^*(K_{\Sigma(4)}+\frac12D)\\
&=&d^*(-2C_0-6f+2(C_0+4f))=d^*(2f)=2F,
\end{eqnarray*}
where $F=d^*(f)$ is a generic fiber of the elliptic fibration.  
Let $\pi: E(4)\to \wh X_4$ be the map given by collapsing the curve $C_2$ to an orbifold point of type $\frac 14(1,1),$ and let $f'=\pi_*(F).$ 
Then the canonical divisor of $\wh X_4$ is $K_{\wh X_4}=2f'.$
The pull-back of $f'$ is $\pi^*(f')=F+\frac14C_2,$ which implies that 
$$f'^2=(\pi^*(f'))^2=F^2+\frac12F\cdot C_2+\frac1{16}C_2^2=0+\frac12-\frac14=\frac14.$$ 
As we can choose a representative of $f'$ to pass through any point of $\wh X_4$, 
the Nakai-Moishezon criterion implies that $f'$ is an ample divisor and as a consequence $K_{\wh X_4}$ is ample. 
Moreover, we have  $K_{\wh X_4}^2=4 f'^2=1.$

For $n\geq5,$ we use the complex structure on $E(n)$ given by the construction in \cite{lp2}. 
 In the Hirzebruch surface $\Sigma(n)=\PP(\OO_{\bcp^1}\oplus\OO_{\bcp^1}(n)),$ 
 let $f_0$ be a fixed fiber and $C_0$ the negative section with $C_0^2=-n.$  
 Lee and Park show \cite{lp2} that  the linear system $|4(C_0+nf)|$ contains a complex divisor, 
 $D,$ which has exactly two singular  points $p,$ and $q,$ both lying on $f_0.$ 
 If $(x,y)$ are the local coordinates at $p$ (respectively at $q$), 
 such that the fiber $f_0$ is given by $(x=0),$ 
 then the equation of $D$ is given by $(y-x)(y+x)=0$ ( $(y-x^{n-4})(y+x^{n-4})=0,$ respectively).
 We resolve the singularities of $D$ by blowing up $\Sigma(n)$ at $p$ once, and 
 $(n-4)$ times at $q.$
 We denote by $\Delta$ the proper transform of $D.$
 To be consistent with the notation in \cite{lp2}, we denote the first blow-up by $E_1,$ 
 and proper transform of the blow-ups at $q$ by $U_{n-5}, \dots, U_{1}, E_2.$ 
 Notice that $E_{1,2}$ are self-intersection $(-1)$ rational curves, while $U_i, i=1,\dots,n-5$ have self-intersection $(-2).$ 
 We denote by $U_{n-4}$ the proper transform of the special fiber $f_0.$
 This is also a self-intersection $(-2)$ curve.
 Let $\pi:Z_n\to\Sigma(n)$ denote the above iterated blow-up of $\Sigma(n)$ at $(n-3)$ points. 
 Then we have the following computations for the canonical divisors:
 \begin{eqnarray*}
K_{\Sigma(n)}&=&-2C_0-(n+2)f\\
K_{Z_n}&=&\pi^*(K_{\Sigma(n)})+E_1+U_{n-5}+2U_{n-6}+\dots+(n-5)U_1+(n-4)E_2.
\end{eqnarray*}
Moreover, we have:
\begin{eqnarray*}
f&=&U_{n-4}+\dots+U_1+E_1+E_2\\
\Delta&=&4(U_{n-4}+nf)-2[E_1+(n-4)E_2+(n-5)U_1+\dots+2U_{n-6}+U_{n-5}],
\end{eqnarray*}
where $f$ is the class of the proper transform of the generic fiber of $\Sigma(n)$ in $Z_n.$
The double cover of $Z_n$ branched along $\Delta$ is the elliptic surface $E(n)$.  
Let $d:Z_n\to E(n)$ denote the branch cover, and $F$ the class of the proper transform of a generic fiber. 
The canonical divisor of $E(n)$ is
\begin{eqnarray}
K_{E(n)}&=&d^*(K_{Z_n}+\frac12\Delta)\notag\\
&=& d^*((n-2)f)	=(n-2)F	\notag
\end{eqnarray}

As the branch locus is disjoint from the chain of curves $C_0, U_{n-4},\cdots,U_1,$ 
its preimage consists of two chains of curves which we denote by 
$C_0', U_{n-4}',\cdots,U_1',$ and $C_0'', U_{n-4}'',\cdots,U_1''.$ 
By Artin's criterion we can collapse the chain $C_0', U_{n-4}',\cdots,U_1',$ and  the chain $ U_{n-4}'',\cdots,U_1''$ 
to obtain an orbifold surface with two singular points, which we denote by $\wh X_n$. 
Then the Hirzebruch-Jung continued fraction tells us that the singularities are of type: 
$\frac1{(n-2)^2}(1,n-3)$ and $\frac1{n-3}(1,n-4),$ respectively. 
Let $\wh\pi:E(n)\to\wh X_n$ be the projection map. 
The projection formula implies that:
$$K_{\wh X_n}=(n-2)\wh F,$$
where we denote by $\wh F=\wh\pi(F).$
We can compute the self-intersection of $\wh F,$ by first computing 
\begin{eqnarray*}
\wh \pi^{*}(\wh F)&=&F+ a_0C_0'+a_1U_1'+\cdots+a_{n-4}U_{n-4}'+b_1U_1''+\cdots+a_{n-4}U_{n-4}''\\
&=&F+\frac1{(n-2)^2}[(n-3)C_0'+(n-4)U_1'+\cdots+1\cdot U_{n-4}'],
\end{eqnarray*}
where the above coefficients are obtained by taking the intersection with the curves in the exceptional locus.
Hence:
$$(\wh F)^2=(\wh \pi^*(\wh F))^2=(\wh \pi^*(\wh F))\cdot F=\frac1{(n-2)^2}(n-3)$$
and 
\begin{equation}
K_{\wh X_n}^2=(n-2)^2\wh F^2=n-3
\end{equation}

We use next the Nakai-Moishezon criterion to show that $K_{\wh X_n}$ is ample. 
As $\wh F$ is the proper transform of a generic fiber on $\Sigma(n),$ it will intersect any curve 
which does not lie on the total transform of the special fiber $f_0$ in a strictly positive number of points. 
So all we need to show is that the intersection of $\wh F$ with the irreducible components 
in the class of the total transform of $f_0$ is strictly positive. 
These components are $\wh\pi(C_0''), \wh\pi(d^{-1}(E_1)),\wh\pi(d^{-1}(E_2)),$ which, 
for simplicity, we denote by $\wh C_0, \wh E_1, \wh E_2.$ 
We have:

\begin{eqnarray}
\wh F\cdot \wh C_0&=&\wh \pi^{*}(\wh F)\cdot\wh \pi^{*}(\wh C_0)=\wh \pi^{*}(\wh F)\cdot C_0''=F\cdot C_0''=1>0\notag\\
\wh F\cdot \wh E_1&=&\wh \pi^{*}(\wh F)\cdot\wh \pi^{*}(\wh E_1)=\wh \pi^{*}(\wh F)\cdot E_1''=\frac1{(n-2)^2}U_{n-4}'\cdot E_1'=\frac1{(n-2)^2}>0\notag\\
\wh F\cdot \wh E_2&=&\wh \pi^{*}(\wh F)\cdot\wh \pi^{*}(\wh E_2)=\wh \pi^{*}(\wh F)\cdot E_2''=(\frac1{(n-2)^2} (n-4)U_1')\cdot E_2''\notag\\
&=&\frac1{(n-2)^2} (n-4)>0\notag
\end{eqnarray}
Hence, the canonical sheaf $K_{\wh X_n}=(n-2)\wh F$ is ample. 
\end{proof}

\begin{prop}
\label{exmp2}
The manifold $Y_n$ obtained by taking the rational blowdown of one chain of type $(-n,-2,\dots,-2)$ 
in the elliptic surface $E(n)$ is a simply connected, minimal symplectic, non-K\"ahler manifold which lies on the half-Noether line.
\end{prop}
\begin{proof}
Propositions \ref{orb-ampl}  and \ref{prop22} imply that $Y_n$ is a minimal symplectic surface of general type. 
Then by Proposition \ref{prop1} its topological invariants are $c_1^2(Y_N)=n-3$ and $c_2(Y_N)=11n+3,$ 
and hence $\chi_h(Y_n)=n.$
Then $Y_n$ lies on the half-Noether line $c_1^2(Y_n)=\chi_h(Y_n)-3,$ and hence it is non-K\"ahler.

The fundamental group of the boundary $\pi_1(L(n^2,nm-1))=\ZZ/n^2\ZZ$ surjects into the fundamental group of the
rational homology ball $\pi_1(B_{n,m})=\ZZ/n\ZZ.$ 
We can then find a loop, which represents a generator of the fundamental group of $B_{n,m},$ which is embedded in the
proper transform of the exceptional divisor $E_1,$ as the curve is transversal to the  collapsed chain.
Hence, the loop is homotopically equivalent to a constant loop in the complement of the configuration of spheres which is rational blowdown in $E(n).$
As $E(n)$ is simply connected then so must be $Y_n.$
\end{proof}

\begin{rmk} An interesting fact is that if one takes the rational blowdown of both chains of type $(-n, -2,\dots,-2)$ in $E(n),$ then the
new manifold lies on the Noether line and admits a K\"ahler structure \cite{lp2}.
\end{rmk}
\bigskip
Theorem \ref{main2} follows immediately from Theorem \ref{main} and Proposition \ref{exmp2}.


\subsection*{Acknowledgements} The author would like to thank Claude LeBrun for interesting discussions on the subject. While writing this paper, the author was partially supported by the NSF grant DMS-1007114. The author gratefully acknowledges support from the Simons Center for Geometry and Physics, Stony Brook University and Institut des Hautes \'Etudes Scientifiques at which some of the research for this paper was performed.

\bibliographystyle{alpha} 

\end{document}